\documentclass{article}
\usepackage[margin=4cm]{geometry}

\usepackage{amsthm}
\usepackage{epsfig,latexsym}
\usepackage{algorithm,algorithmicx,algpseudocode}
\usepackage{xspace}
\usepackage{amsmath,amssymb}
\usepackage{graphicx}
\usepackage{caption}
\usepackage{subcaption}
\usepackage{mathrsfs}
\usepackage{bbm}
\usepackage{color}
\usepackage{mhchem}
\usepackage{tikz}
\usepackage{pgfplots}
\usepackage{xfrac}

\newcommand{\dd}{\,\mathrm{d}}

\let\originalleft\left
\let\originalright\right
\renewcommand{\left}{\mathopen{}\mathclose\bgroup\originalleft}
\renewcommand{\right}{\aftergroup\egroup\originalright}

\newcommand{\EV}[1]{\mathbb{E}\left[#1\right]}
\newcommand{\VV}[1]{\mathbb{V}\left[#1\right]}
\newcommand{\cost}[1]{\mathcal{C}\left[#1\right]}
\newcommand{\rmse}[1]{e\left[#1\right]}
\newcommand{\QMC}[1]{\widehat Q^{MC}_{N_\ell,#1}}

\newcommand{\QML}[1]{\widehat Q^{ML}_{\{N_\ell\},#1}}

\newcommand{\lip}{C_{L}}

\newcommand{\domain}{\mathscr{D}}
\newcommand{\ME}{C_1}
\newcommand{\MV}{C_2}

\newtheorem{theorem}{Theorem}
\newtheorem{lemma}[theorem]{Lemma}

\newtheorem{example}[theorem]{Example}
\newtheorem{assumption}[theorem]{Assumption}
\usepackage[utf8]{inputenc}

\title{A multilevel Monte Carlo method for computing failure probabilities}

\author{Daniel Elfverson\footnotemark[1] \and Fredrik Hellman\footnotemark[2] \and Axel Målqvist\footnotemark[3]\ }

\begin{document}

\renewcommand{\thefootnote}{\fnsymbol{footnote}}
\footnotetext[1]{Information Technology, Uppsala University, Box 337, SE-751 05, Uppsala, Sweden (\texttt{daniel.elfverson@it.uu.se}). Supported by the Göran Gustafsson Foundation.}
\footnotetext[2]{Information Technology, Uppsala University, Box 337, SE-751 05, Uppsala, Sweden (\texttt{fredrik.hellman@it.uu.se}). Supported by the Centre of Interdisciplinary Mathematics, Uppsala University.}
\footnotetext[3]{Department of Mathematical Sciences, Chalmers University of Technology and University of Gothenburg, SE-412 96 Göteborg, Sweden (\texttt{axel@chalmers.se}). Supported by the Swedish Research Council.}
\renewcommand{\thefootnote}{\arabic{footnote}}

\maketitle

\begin{abstract}
  We propose and analyze a method for computing failure probabilities
  of systems modeled as numerical deterministic models (e.g., PDEs)
  with uncertain input data. A failure occurs when a functional of the
  solution to the model is below (or above) some critical value.  By
  combining recent results on quantile estimation and the multilevel
  Monte Carlo method we develop a method which reduces computational
  cost without loss of accuracy. We show how the computational cost of
  the method relates to error tolerance of the failure
  probability. For a wide and common class of problems, the
  computational cost is asymptotically proportional to solving a
  single accurate realization of the numerical model, i.e.,
  independent of the number of samples. Significant reductions in
  computational cost are also observed in numerical experiments.
\end{abstract}



\pagestyle{myheadings}

\thispagestyle{plain}


\section{Introduction}
This paper is concerned with the computational problem of finding the
probability for failures of a modeled system. The model input is
subject to uncertainty with known distribution and a failure is the
event that a functional (quantity of interest, QoI) of the model
output is below (or above) some critical value. The goal of this paper
is to develop an efficient and accurate multilevel Monte Carlo (MLMC)
method to find the failure probability. We focus mainly on the case
when the model is a partial differential equation (PDE) and we use
terminology from the discipline of numerical methods for
PDEs. However, the methodology presented here is also applicable in a
more general setting.

A multilevel Monte Carlo method inherits the non-intrusive and
non-parametric characteristics from the standard Monte Carlo (MC)
method. This allows the method to be used for complex black-box
problems for which intrusive analysis is difficult or impossible. The
MLMC method uses a hierarchy of numerical approximations on different
accuracy levels. The levels in the hierarchy are typically directly
related to a grid size or timestep length. The key idea behind the
MLMC method is to use low accuracy solutions as control variates for
high accuracy solutions in order to construct an estimator with lower
variance. Savings in computational cost are achieved when the low
accuracy solutions are cheap and are sufficiently correlated with the
high accuracy solutions. MLMC was first introduced in \cite{Gil08} for
stochastic differential equations as a generalization of a two-level
variance reduction technique introduced in \cite{Keb05}. The method
has been applied to and analyzed for elliptic PDEs in
\cite{ClGiScTe11, ChScTe13, TeScGiUl13}. Further improvements of the
MLMC method, such as work on optimal hierarchies, non-uniform meshes
and more accurate error estimates can be found in \cite{LNST14,
  CHNST14}. In the present paper, we are not interested in the
expected value of the QoI, but instead a failure probability, which is
essentially a single point evaluation of the cumulative distribution
function (cdf). For extreme failure probabilities, related methods
include importance sampling \cite{Gly96}, importance splitting
\cite{GlHeShZa96}, and subset simulations \cite{AuBe01}. Works more
related to the present paper include non-parameteric density
estimation for PDE models in \cite{EsMaTa09}, and in particular
\cite{EEHM14}. In the latter, the selective refinement method for
quantiles was formulated and analyzed.

In this paper, we seek to compute the cdf at a given critical
value. The cdf at the critical value can be expressed as the
expectation value of a binomially distributed random variable $Q$ that
is equal to $1$ if the QoI is smaller than the critical value, and $0$
otherwise. The key idea behind selective refinement is that
realizations with QoI far from the critical value can be solved to a
lower accuracy than those close to the critical value, and still yield
the same value of $Q$. The random variable $Q$ lacks regularity with
respect to the uncertain input data, and hence we are in an
unfavorable situation for application of the MLMC method. However,
with the computational savings from the selective refinement it is
still possible to obtain an asymptotic result for the computational
cost where the cost for the full estimator is proportional to the cost
for a single realization to the highest accuracy.

The paper is structured as
follows. Section~\ref{sec:problem_formulation} presents the necessary
assumptions and the precise problem description. It is followed by
Section~\ref{sec:failure} where our particular failure probability
functional is defined and analyzed for the MLMC method. In
Section~\ref{sec:MLMC} and Section~\ref{sec:selective} we revisit the
multilevel Monte Carlo and selective refinement method adapted to this
problem and in Section \ref{sec:SMLMC} we show how to combine
multilevel Monte Carlo with the selective refinement to obtain optimal
computational cost. In Section~\ref{sec:heuristic} we give details on
how to implement the method in practice. The paper is concluded with
two numerical experiments in Section~\ref{sec:numerical}.

\section{Problem formulation}
\label{sec:problem_formulation}
We consider a model problem $\mathcal{M}$, e.g., a (non-)linear
differential operator with uncertain data. We let $u$ denote the
solution to the model
\begin{equation*}\label{eq:model}
  \mathcal{M}(\omega,u) = 0,
\end{equation*}
where the data $\omega$ is sampled from a space $\Omega$. In what
follows we assume that there exists a unique solution $u$ given any
$\omega\in\Omega$ almost surely. It follows that the solution $u$ to a
given model problem $\mathcal{M}$ is a random variable which can be
parameterized in $\omega$, i.e., $u=u(\omega)$.

The focus of this work is to compute failure probabilities, i.e., we
are not interested in some pointwise estimate of the expected value of
the solution, $\EV{u}$, but rather the probability that a given QoI
expressed as a functional, $X(u)$ of the solution $u$, is less
(or greater) than some given critical value $y$.
We let $F$ denote the cdf of the random variable $X = X(\omega)$. The
failure probability is then given by
\begin{equation}\label{eq:failure}
  p=F(y) = \Pr(X \le y).
\end{equation}
The following example illustrates how the problem description relates
to real world problems.
\begin{example}
  As an example, geological sequestration of carbon dioxide
  (\ce{CO_2}) is performed by injection of \ce{CO_2} in an underground
  reservoir. The fate of the \ce{CO_2} determines the success or
  failure of the storage system. The \ce{CO_2} propagation is often
  modeled as a PDE with random input data, such as a random permeability
  field. Typical QoIs include reservoir breakthrough time or pressure
  at a fault. The value $y$ corresponds to a critical value which the
  QoI may not exceed or go below. In the breakthrough time case, low
  values are considered failure. In the pressure case, high values are
  considered failure. In that case one should negate the QoI to transform the
  problem to the form of equation \eqref{eq:failure}.
\end{example}

The only regularity assumption on the model is the following Lipschitz
continuity assumption of the cdf, which is assumed to hold throughout the paper.
\begin{assumption}
  \label{ass:lip}
  For any $x,y\in\mathbb{R}$,
  \begin{equation}
      |F(x)-F(y)| \leq \lip |x-y|.
  \end{equation}
\end{assumption}

To compute the failure probability we consider the binomially
distributed variable $Q = \mathbbm{1}{(X\le y)}$ which takes the value
$1$ if $X\le y$ and $0$ otherwise. The cdf can be expressed as the
expected value of $Q$, i.e., $p=F(y)= \EV{Q}$. In practice we
construct an estimator $\widehat{Q}$ for $\EV{Q}$, based on
approximate sample values from $X$. As such, $\widehat Q$ often
suffers from numerical bias from the approximation in the underlying
sample. Our goal is to compute the estimator $\widehat Q$ to a given
root mean square error (RMSE) tolerance $\epsilon$, i.e., to compute
\begin{equation*}
  \label{eq:rmse}
  \rmse{\widehat Q}
  = \left(\EV{\left(\widehat Q- \EV{Q}\right)^2}\right)^{1/2} =\left(\VV{\widehat Q } + \left(\EV{\widehat Q- Q}\right)^2\right)^{1/2} \leq \epsilon
\end{equation*}
to a minimal computational cost. The equality above shows a standard
way of splitting the RMSE into a stochastic error and numerical bias
contribution.

The next section presents assumptions and results regarding the
numerical discretization of the particular failure probability
functional $Q$. 

\section{Approximate failure probability functional}
\label{sec:failure}
We will not consider a particular approximation technique for
computing $\widehat Q$, but instead make some abstract assumptions on
the underlying discretization. We introduce a hierarchy of refinement
levels $\ell = 0,1,\ldots$ and let $X'_\ell$ and $Q'_\ell =
\mathbbm{1}{(X'_\ell \le y)}$ be an approximate QoI of the model, and
approximate failure probability, respectively, on level $\ell$. One
possible and natural way to define the accuracy on level $\ell$ is by
assuming
\begin{equation}
\label{eq:uniformerror}
\left|X - X'_\ell\right| \le \gamma^\ell,
\end{equation}
for some $0 < \gamma < 1$. This means the error of all realizations on
level $\ell$ are uniformly bounded by $\gamma^\ell$. In a PDE setting,
typically an a priori error bound or a posteriori error estimate,
\begin{equation*}
\left|X(\omega) - X_h(\omega)\right| \le C(\omega) h^{s},
\end{equation*}
can be derived for some constants $C(\omega)$, $s$, and a
discretization parameter $h$. Then we can choose $X'_\ell = X_h$ with
$h = \left(C(\omega)^{-1}\gamma^\ell\right)^{1/s}$ to fulfill
\eqref{eq:uniformerror}.

For an accurate value of the failure probability functional the
condition in \eqref{eq:uniformerror} is unnecessarily strong. This
functional is very sensitive to perturbations of values close to $y$,
but insensitive to perturbations for values far from $y$. This
insensitivity can be exploited. We introduce a different approximation
$X_\ell$, and impose the following, relaxed, assumption on this
approximation of $X$, which allows for larger errors far from the
critical value $y$. This assumption is illustrated in
Figure~\ref{fig:numericalerror}.
\begin{assumption}
  \label{ass:qoierror}
  The numerical approximation $X_\ell$ of $X$ satisfies
  \begin{equation}\label{eq:numerror}
    \left|X - X_\ell\right| \le \gamma^\ell \quad \text{ or } \quad \left|X - X_\ell\right| < \left|X_\ell - y\right|
  \end{equation}
  for a fix $0 < \gamma < 1$.
\end{assumption}
\begin{figure}[h!]
  \centering
  \begin{tikzpicture}
    \begin{axis}[
      width=7cm,
      xmin=0,
      xmax=5,
      xlabel={$X_\ell$},
      xtick=\empty,
      extra x ticks={2.5},
      extra x tick style={xticklabel = $y$},
      ymin=0,
      ymax=3,
      ytick=\empty,
      extra y ticks={0.5},
      extra y tick style={yticklabel = $\gamma^\ell$},
      ylabel={$|X - X_\ell|$},
      axis lines=middle,
      axis line style=->,
      legend style={
        at={(1.1,0.5)},
        anchor={west},
        cells={anchor=west},
      }
      ]
      \addplot[domain=2:3] {0.5};
      \addplot[domain=0:2,dashed] {-x+2.5};
      \addplot[domain=2:2.5,dotted] {-x+2.5};
      \addplot[domain=0:2,dotted] {0.5};
      \addplot[domain=3:5,dotted] {0.5};
      \addplot[domain=3:5,dashed] {x-2.5};
      \addplot[domain=2.5:3,dotted] {x-2.5};
      \addlegendentry{$|X-X_\ell| \le \gamma^\ell$};
      \addlegendentry{$|X-X_\ell| < |X_\ell - y|$};
    \end{axis}
  \end{tikzpicture}
  \caption{Illustration of condition \eqref{eq:numerror}. The numerical error is
    allowed to be larger than $\gamma^\ell$ far away from $y$.}
  \label{fig:numericalerror}
\end{figure}
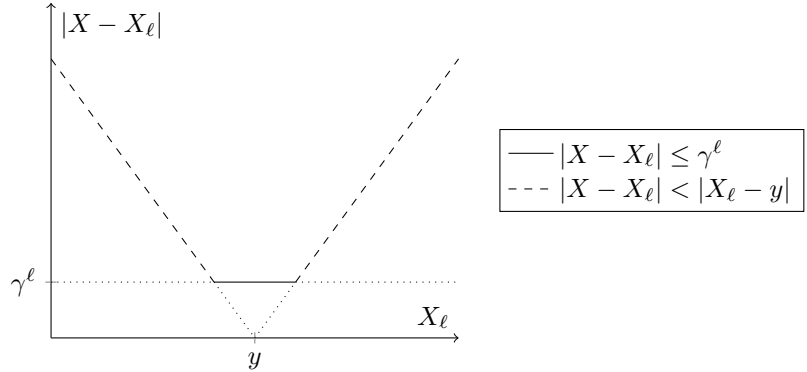

We define $Q_\ell = \mathbbm{1}{(X_\ell \le y)}$ analogously to
$Q'_\ell$. Let us compare the implications of the two conditions
\eqref{eq:uniformerror} and \eqref{eq:numerror} on the quality of the
two respective approximations. Denote by $X'_\ell$ and $Q'_\ell$
stochastic variables obeying the error bound \eqref{eq:uniformerror}
and its corresponding approximate failure functional, respectively,
and let $X_\ell$ obey \eqref{eq:numerror}. In a practical situation,
Assumption~\ref{ass:qoierror} is fulfilled by iterative refinements of
$X_\ell$ until condition \eqref{eq:numerror} is satisfied. It is
natural to use a similar procedure to achieve the stricter condition
\eqref{eq:uniformerror} for $X'_\ell$. We express this latter
assumption of using similar procedures for computing $X_\ell$ and
$X'_\ell$ as
\begin{equation}
  \label{eq:similarmethods}
  |X - X_\ell| \le \gamma^\ell \text{ implies } X'_\ell = X_\ell,
\end{equation}
i.e., for outcomes where $X_\ell$ is solved to accuracy
$\gamma^\ell$, $X'_\ell$ is equal to $X_\ell$. Under that assumption,
the following lemma shows that it is not less probable that $Q_\ell$
is correct than that $Q'_\ell$ is.
\begin{lemma}
  Let $X'_\ell$ and $X_\ell$ fulfill \eqref{eq:uniformerror} and
  \eqref{eq:numerror}, respectively, and assume
  \eqref{eq:similarmethods} holds. Then $\Pr(Q_\ell = Q) \ge
  \Pr(Q'_\ell = Q)$.
\end{lemma}
\begin{proof}
  We split $\Omega$ into the events $A = \{\omega \in \Omega : |X -
  X_\ell| \le \gamma^\ell\}$ and its complement $\Omega \setminus
  A$. For $\omega \in A$, using \eqref{eq:similarmethods}, we conclude
  that $Q'_\ell = Q_\ell$, hence
  \begin{equation*}
    \Pr(Q_\ell = Q \;|\; A) = \Pr(Q'_\ell = Q \;|\; A).
  \end{equation*}
  For $\omega \notin A$, we have $|X - X_\ell| > \gamma^\ell$, and
  from \eqref{eq:numerror} that $|X - X_\ell| < |X_\ell - y|$, i.e.,
  $Q_\ell = Q$ and hence
  \begin{equation*}
    \Pr(Q_\ell = Q \;|\; \Omega \setminus A) = 1.
  \end{equation*}
  Since $\Pr(Q'_\ell = Q \;|\; \Omega \setminus A) \le 1$, we get
  $\Pr(Q_\ell = Q) \ge \Pr(Q'_\ell = Q)$.
\end{proof}

Under Assumption~\ref{ass:qoierror} we can prove the following lemma
on the accuracy of the failure probability function $Q_\ell$.
\begin{lemma}\label{lem:assEVW}
  Under Assumption~\ref{ass:lip} and \ref{ass:qoierror}, the statements
  \begin{description}
  \item[M1] $\left|\EV{Q_\ell-Q}\right|\leq \ME\gamma^\ell$,
  \item[M2] $\VV{Q_\ell-Q_{\ell-1}} \leq
    \MV\gamma^{\ell}$ for $\ell \geq 1$,
  \end{description}
  are satisfied where $\ME$ and $\MV$ do not depend on $\ell$.
\end{lemma}
\begin{proof}
  \label{proof:assEVW1}
  We split $\Omega$ into the events $B = \{\omega \in \Omega :
  \gamma^\ell \ge |X_\ell - y|\}$ and its complement $\Omega \setminus
  B$. In $\Omega \setminus B$, we have $Q_\ell = Q$, since $|X -
  X_\ell| < |X_\ell - y|$ from \eqref{eq:numerror}. Also, we note that
  the event $B$ implies $|X - X_\ell| \le \gamma^\ell$, hence $|X - y|
  \le 2\gamma^\ell$. Then,
  \begin{equation*}
    \begin{aligned}
    |\EV{Q_\ell - Q}| & = \left|\int_{B} Q_\ell(\omega) - Q(\omega)\,\mathrm{d}P(\omega)\right|
     \le \int_{B} 1\,\mathrm{d}P(\omega) \\
    & \le \Pr(|X - y| \le 2\gamma^\ell)
     = F(y - 2\gamma^\ell) - F(y + 2\gamma^\ell) \\
    & \le 4C_L\gamma^\ell, \\
    \end{aligned}
  \end{equation*}
  which proves
  \textbf{M1}. \textbf{M2} follows directly from \textbf{M1}, since
  \begin{equation*}
    \begin{aligned}
      \VV{Q_\ell-Q_{\ell-1}} & = \EV{(Q_\ell-Q_{\ell-1})^2} - \EV{Q_\ell-Q_{\ell-1}}^2 \\
      & \le \EV{Q_\ell-2Q_\ell Q_{\ell-1} + Q_{\ell-1}} \\
      & \le \left|\EV{Q_\ell-Q}\right| + \left|2\EV{Q_\ell Q_{\ell-1}-Q}\right| + \left|\EV{Q_{\ell-1}-Q}\right| \\
      & \le 2\left|\EV{Q_\ell-Q}\right| + 2\left|\EV{Q_{\ell-1}-Q}\right| \\
      & \le C_2\gamma^\ell, \\
    \end{aligned}
  \end{equation*}
  where $(Q_\ell)^2 = Q_\ell$ was used.
\end{proof}

Interesting to note with this particular failure probability
functional is that the convergence rate in \textbf{M2} cannot be
improved if the rate in \textbf{M1} is already sharp, as the following
lemma shows.
\begin{lemma}
  Let $0 < \gamma < 1$ be fixed. If there is a $0<c\leq 1$ such that
  the failure probability functional satisfies
  \begin{equation*}
    c\gamma^\ell \leq|\EV{Q_\ell - Q}| \leq \ME \gamma^\ell
  \end{equation*}
  for all $\ell=0,\ldots$, then
  \begin{equation*}
    \VV{Q_\ell - Q_{\ell-1}} \leq \MV \gamma^{\beta\ell},
  \end{equation*}
   where $\beta = 1$ is sharp in the sense that the relation will be violated for sufficiently large $\ell$, if $\beta>1$.
\end{lemma}
\begin{proof}
  Assume that $\VV{Q_\ell - Q_{\ell-1}}\leq C\gamma^{\beta\ell}$ for
  for some constant $C$ and $\beta>1$. For two levels $\ell$ and
  $k$, we have that
  \begin{equation*}
    \begin{aligned}
      |\EV{Q_\ell - Q_k}| &\geq \left| |\EV{Q_\ell-Q}|-|\EV{Q_k-Q}|\right|  \geq \left(c-\gamma^{\ell-k}\right)\gamma^k.
    \end{aligned}
  \end{equation*}
  Choosing $\ell$ and $k$ such that $\ell > k$ and
  $c-\gamma^{\ell-k}>0$ yields
   \begin{equation*}
    \begin{aligned}
      \left(c-\gamma^{\ell-k}\right)\gamma^k &\leq |\EV{Q_\ell - Q_k}| \leq \sum_{j=k}^{\ell-1}|\EV{Q_{j+1} - Q_j}| \leq \sum_{j=k}^{\ell-1}\EV{(Q_{j+1} - Q_j)^2} \\
      & = \sum_{j=k}^{\ell-1}\left(\VV{Q_{j+1} - Q_j} + \left(\EV{Q_{j+1} - Q_j}\right)^2\right) \\
      & \leq  \sum_{j=k}^{\ell-1}\left(C\gamma^{\beta j} + \mathcal{O}(\gamma^{2j})\right) \le \widetilde C\gamma^{\beta k} + \mathcal{O}(\gamma^{2k}).
    \end{aligned}
 \end{equation*}
 For $\ell,k \to \infty$ we have a contradiction and hence $\beta\leq 1$,
 which proves that the bound can not be improved.
\end{proof}

\section{Multilevel Monte Carlo method}\label{sec:MLMC}
In this section, we present the multilevel Monte Carlo method in a
general context. Because of the low convergence rate of the variance
in \textbf{M2}, the MLMC method does not perform optimally for
the failure probability functional.  The results presented here will
be combined with the results from Section~\ref{sec:selective} to
derive a new method to compute failure probabilities efficiently in
Section~\ref{sec:SMLMC}.

The (standard) MC estimator at refinement level $\ell$ of $\EV{Q}$
using a sample $\{\omega^{i}_\ell\}_{i=1}^{N_\ell}$, reads
\begin{equation*}\label{eq:MC}
  \QMC{\ell} = \frac{1}{N_\ell}\sum_{i=1}^{N_\ell}Q_\ell(\omega^{i}_\ell).
\end{equation*}
Note that the subscripts $N_\ell$ and $\ell$ control the statistical
error and numerical bias, respectively. The expected value and
variance of the estimator $\QMC{\ell}$ are
$\EV{\QMC{\ell}}=\EV{Q_\ell}$ and $\VV{\QMC{\ell}} =
N_\ell^{-1}\VV{Q_\ell}$, respectively. Referring to the goal of the
paper, we want the MSE (square of the RMSE) to satisfy
\begin{equation*}
  \label{eq:MCerror}
  \rmse{\QMC{\ell}}^2 = N_\ell^{-1}\VV{Q_\ell} + \left(\EV{Q_\ell-Q}\right)^2\leq \epsilon^2/2 + \epsilon^2/2=\epsilon^2,
\end{equation*}
i.e., both the statistical error and the numerical error should be
less than $\epsilon^2/2$. The MLMC method is a variance reduction
technique for the MC method. The MLMC estimator $\QML{L}$ at
refinement level $L$ is expressed as a telescoping sum of $L$ MC
estimator correctors:
\begin{equation*}
  \QML{L} = \sum_{\ell=0}^L\frac{1}{N_\ell}\sum_{i=1}^{N_\ell}\left(Q_\ell(\omega^{i}_\ell) - Q_{\ell-1}(\omega^{i}_\ell)\right),
\end{equation*}
where $Q_{-1}=0$. There is one corrector for every
refinement level $\ell = 0,\ldots,L$, each with a specific MC
estimator sample size $N_\ell$. The expected value and variance of the
MLMC estimator are
\begin{equation}\label{eq:MLCMexpvar}
  \begin{aligned}
    \EV{\QML{L}} &= \sum_{\ell=0}^L\EV{Q_\ell-Q_{\ell-1}} = \EV{Q_L}\quad \text{and} \\
    \VV{\QML{L}} &= \sum_{\ell=0}^LN_\ell^{-1}\VV{Q_\ell-Q_{\ell-1}},
  \end{aligned}
\end{equation}
respectively. Using \eqref{eq:MLCMexpvar} the MSE for the MLMC
estimator can be expressed as
\begin{equation*}
  \begin{aligned}
    \rmse{\QML{L}}^2 &= \sum_{\ell=0}^LN_\ell^{-1} \VV{Q_\ell-Q_{\ell-1}}+\left(\EV{Q_L-Q}\right)^2,
  \end{aligned}
\end{equation*}
 and can be computed at expected cost
\begin{equation*}
  \cost{\QML{L}} = \sum_{\ell=0}^L N_\ell c_\ell,
\end{equation*}
where $c_\ell=\cost{Q_\ell}+\cost{Q_{\ell-1}}$. Here, by
$\cost{\cdot}$ we denote the expected computational cost to compute a
certain quantity. Given that the variance of the MLMC estimator is
$\epsilon^2/2$ the expected cost is minimized by choosing
\begin{equation}\label{eq:MLMCoptimalNl}
  N_\ell = 2\epsilon^{-2}\sqrt{\VV{Q_\ell-Q_{\ell-1}}/ c_\ell}\sum_{k=0}^L\sqrt{\VV{Q_{k}-Q_{k-1}}c_k}
\end{equation}
(see Appendix \ref{optimalNl}), and hence the total expected cost is
\begin{equation}
  \label{eq:MLMCtotalcost}
  \cost{\QML{L}} = 2\epsilon^{-2}\left(\sum_{\ell=0}^L \sqrt{\VV{Q_\ell-Q_{\ell-1}}c_\ell}\right)^2.
\end{equation}
If the product $\VV{Q_\ell-Q_{\ell-1}}c_\ell$ increases (or decreases)
with $\ell$ then dominating term in \eqref{eq:MLMCtotalcost} will be
$\ell=L$ (or $\ell = 0$). The values $N_\ell$ can be estimated on the
fly in the MLMC algorithm using \eqref{eq:MLMCoptimalNl} while the
cost $c_\ell$ can be estimated using an a priori model.  The
computational complexity to obtain a RMSE less than $\epsilon$ of the
MLMC estimator for the failure probability functional is given by the
theorem below. In the following, the notation $a \lesssim b$ stands
for $a \le Cb$ with some constant $C$ independent of $\epsilon$ and
$\ell$.
\begin{theorem}
  \label{thm:MLMC}
  Let Assumption~\ref{ass:lip} and \ref{ass:qoierror} hold (so that
  Lemma~\ref{lem:assEVW} holds) and $\cost{Q_\ell} \lesssim \gamma^{-r
    \ell}$. Then there exists a constant $L$ and a sequence
  $\{N_\ell\}$ such that the RMSE is less than $\epsilon$, and the
  expected cost of the MLMC estimator is
  \begin{equation}
    \cost{\QML{L}} \lesssim
    \begin{cases}
      \epsilon^{-2} & \quad r < 1\\
      \epsilon^{-2}(\log\epsilon^{-1})^2 & \quad r = 1 \\
      \epsilon^{-1-r} & \quad r > 1.  \\
    \end{cases}
  \end{equation}
\end{theorem}
\begin{proof}
  For a proof see, e.g., \cite{ClGiScTe11,Gil08}.
\end{proof}

The most straight-forward procedure to fulfill
Assumption~\ref{ass:qoierror} in practice is to refine all samples on
level $\ell$ uniformly to an error tolerance $\gamma^\ell$, i.e., to
compute $X'_\ell$ introduced in Section~\ref{sec:failure}, for which
$\left|X - X'_\ell\right| \leq \gamma^{\ell}$. Typical numerical
schemes for computing $X'_\ell$ include finite element, finite volume,
or finite difference schemes. Then the expected cost
$\cost{Q'_\ell}$ typically fulfill
\begin{equation}
\label{eq:costmodel}
\cost{Q'_\ell} = \gamma^{-q\ell},
\end{equation}
where $q$ depends on the physical dimension of the computational
domain, the convergence rate of the solution method, and computational
complexity for assembling and solving the linear system. Note that one
unit of work is normalized according to equation (\ref{eq:costmodel}).
Using Theorem~\ref{thm:MLMC}, with $Q'_\ell$ instead of $Q_\ell$
(which is possible, since $Q'_\ell$ trivially fulfills
Assumption~\ref{ass:qoierror}) we obtain a RMSE of the expected cost
less than $\epsilon^{-1-q} = \epsilon^{-1}\cost{Q'_\ell}$
for the case $q > 1$.

In the next section we describe how the selective refinement
algorithm computes $X_\ell$ (hence $Q_\ell$) that fulfills
Assumption~\ref{ass:qoierror} to a lower cost than its fully refined
equivalent $X'_\ell$. The theorem above can then be applied with $r =
q - 1$ instead of $r = q$.

\section{Selective refinement algorithm}
\label{sec:selective}
In this section we modify the selective refinement algorithm proposed
in \cite{EEHM14} for computing failure probabilities (instead of
quantiles) and for quantifying the error using the RMSE. The selective
refinement algorithm computes $X_\ell$ so that
\begin{equation*}\label{eq:numerror2}
  \left|X - X_\ell\right| \le \gamma^{\ell} \quad \text{ or } \quad \left|X - X_\ell\right| < \left|X_\ell - y\right|
\end{equation*}
in Assumption~\ref{ass:qoierror} is fulfilled without requiring the
stronger (full refinement) condition
\begin{equation*}
  \label{eq:uniformerror2}
    \left|X - X_\ell\right| \le \gamma^\ell.
\end{equation*}
In contrast to the selective refinement algorithm in \cite{EEHM14},
Assumption~\ref{ass:qoierror} can be fulfilled by iterative refinement
of realizations over all realizations independently. This allows for
an efficient totally parallell implementation. We are particularly
interested in quantifying the expected cost required by the selective
refinement algorithm, and showing that the $X_\ell$ resulting from the
algorithm fulfills Assumption~\ref{ass:qoierror}.

Algorithm~\ref{alg:selective} exploits the fact that $Q_\ell=Q$ for
realizations satisfying $|X - X_\ell| < |X_\ell - y|$. That is, even
if the error of $X_\ell$ is greater than $\gamma^\ell$, it might be
sufficiently accurate to yield the correct value of $Q_\ell$.  The
algorithm works on a per-realization basis, starting with an error
tolerance $1$. The realization is refined iteratively until
Assumption~\ref{ass:qoierror} is fulfilled. The advantage is that many
samples can be solved only with low accuracy and hence the average
cost per $Q_\ell$ is
reduced. Lemma~\ref{lem:selective_satisfies_assumption} shows that
$X_\ell$ computed using Algorithm~\ref{alg:selective} satisfies
Assumption~\ref{ass:qoierror}.
\begin{algorithm}[htb]\caption{Selective refinement algorithm}
\label{alg:selective}
\begin{algorithmic}[1]
  \State Input arguments: level $\ell$, realization $i$, critical value $y$, and tolerance factor $\gamma$
  \State Compute $X_\ell(\omega^i_\ell)$ to tolerance $1$
  \State Let $j = 0$
  \While{$j \le \ell$ and $\gamma^j > |X_\ell(\omega^i_\ell) - y|$}
  \State Recompute $X_\ell(\omega^i_\ell)$ to tolerance $\gamma^j$
  \State Let $j = j + 1$
  \EndWhile
  \State Final $X_\ell(\omega^i_\ell)$ is the result
\end{algorithmic}
\end{algorithm}
\begin{lemma}\label{lem:selective_satisfies_assumption}
  Approximations $X_\ell$ computed using Algorithm~\ref{alg:selective}
  satisfy Assumption~\ref{ass:qoierror}.
\end{lemma}

\begin{proof}
  At each iteration in the while-loop of
  Algorithm~\ref{alg:selective}, $\gamma^j$ is the error tolerance of
  $X_\ell(\omega^i_\ell)$, i.e., $|X(\omega^i_\ell) -
  X_\ell(\omega^i_\ell)| \le \gamma^j$. The stopping criterion hence
  implies Assumption~\ref{ass:qoierror} for $X_\ell(\omega^i_\ell)$.
\end{proof}

The expected cost for computing $Q_\ell$ using
Algorithm~\ref{alg:selective} is given by the following lemma.
\begin{lemma}\label{lem:cost_selective}
  The expected cost to compute the failure probability functional
  using Algorithm~\ref{alg:selective} can be bounded as
  \begin{equation*}\label{eq:costselective}
    \cost{Q_\ell} \lesssim \sum_{j=0}^\ell\gamma^{(1-q)j}.
  \end{equation*}
\end{lemma}
\begin{proof}
  Consider iteration $j$, i.e., when $X_\ell(\omega^i_\ell)$ has
  been computed to tolerance $\gamma^{j-1}$. We denote by $E_{j}$ the
  probability that a realization enters iteration $j$. For $j \le
  \ell$,
  \begin{equation*}\label{eq:probToEnter}
    \begin{aligned}
      \Pr(E_{j}) & = \Pr(y-\gamma^{j-1} \le X_\ell \le y + \gamma^{j-1}) \\
      & \le \Pr(y - 2\gamma^{j-1} \le X \le y + 2\gamma^{j-1}) \\
      & = F(y+2\gamma^{j-1}) - F(y-2\gamma^{j-1}) \\
      & \leq 4\lip \gamma^{j-1}.
    \end{aligned}
  \end{equation*}
  Every realization is initially solved to tolerance $1$. Using that
  the cost for solving a realization to tolerance $\gamma^j$ is
  $\gamma^{-qj}$, we get that the expected cost is
  \begin{equation*}\label{eq:costselective1}
    \cost{Q_\ell} = 1 + \sum_{j=1}^\ell \Pr(E_j) \gamma^{-qj} \le 1 + \sum_{j=1}^\ell4\lip\gamma^{j-1}\gamma^{-qj} \lesssim \sum_{j=0}^\ell\gamma^{(1-q)j}
  \end{equation*}
  which concludes the proof.
\end{proof}

\section{Multilevel Monte Carlo using the selective refinement strategy}\label{sec:SMLMC}
Combining the MLMC method with the algorithm for selective refinement
there can be further savings in computational cost. We call this
method multilevel Monte Carlo with selective refinement (MLMC-SR). In
particular, for $q>1$ we obtain from Lemma~\ref{lem:cost_selective}
that the expected cost for one sample can be bounded as
\begin{equation}
  \label{eq:costmodel_selective}
  \begin{aligned}
    \cost{Q_\ell} \lesssim \sum_{j=0}^\ell\gamma^{(1-q)j} \lesssim \gamma^{(1-q)\ell}.
  \end{aligned}
\end{equation}
Applying Theorem~\ref{thm:MLMC} with $r = q-1$ yields the
following result.

\begin{theorem}
  \label{thm:workMLMCsel}
  Let Assumption~\ref{ass:lip} and Assumption~\ref{ass:qoierror} hold
  (so that Lemma~\ref{lem:assEVW} holds) and suppose that
  Algorithm~\ref{alg:selective} is executed to compute $Q_\ell$. Then
  there exists a constant $L$ and a sequence $\{N_\ell\}$ such that the
  RMSE is less than $\epsilon$, and the expected cost for the MLMC
  estimator with selective refinement is
  \begin{equation}
    \cost{\QML{L}} \lesssim
    \begin{cases}
      \epsilon^{-2} & \quad q<2 \\
      \epsilon^{-2}(\log\epsilon^{-1})^2 & \quad q = 2 \\
      \epsilon^{-q} & \quad q>2.
    \end{cases}
  \end{equation}
\end{theorem}
\begin{proof}
  For $q>1$, follows directly from Theorem~\ref{thm:MLMC} since
  Lemma~\ref{lem:assEVW} holds with $r=q-1$. For $q \le 1$, we
  use the rate $\epsilon^{-2}$ from the case $1 < q < 2$, since the
  cost cannot be worsened by making each sample cheaper to
  compute.
\end{proof}

In a standard MC method we have $\epsilon^{-2}\sim N$ where $N$ is the
number of samples and $\epsilon^{-q}\sim \cost{Q'_L}$ where
$\cost{Q'_L}$ is the expected computational cost for solving one
realization on the finest level without selective refinement. The
MLMC-SR method then has the following cost,
\begin{equation}
  \cost{\QML{L}} \lesssim
  \begin{cases}
    N & \quad q< 2 \\
    \cost{Q'_L} & \quad q > 2.
  \end{cases}
\end{equation}
A comparison of MC, MLMC with full refinement (MLMC), and MLMC with
selective refinement (MLMC-SR), is given in
Table~\ref{tab:comparison}. To summarize, the best possible scenario
is when the cost is $\epsilon^{-2}$, which is equivalent with a
standard MC method where all samples can be obtained with cost
$1$. This complexity is obtained for the MLMC method when $q<1$ and for
the MLMC-SR method when $q<2$. For $q>2$ the MC method has the same
complexity as solving $N$ problem on the finest level $N\cost{Q'_L}$,
 MLMC has the same cost as $N^{1/2}$ problem on the finest level
$N^{1/2}\cost{Q'_L}$, and MLMC-SR method as solving one problem on the
finest level $\cost{Q'_L}$.
  \begin{table}[ht]
    \begin{center}
      \begin{tabular}{|c|c|c|c|} \hline
        Method  & $0\leq q<1$  & $1< q < 2$ & $ q>2$  \\ \hline\hline
        MC    & $\epsilon^{-2-q}$  &  $\epsilon^{-2-q}$   &  $\epsilon^{-2-q}$ \\ \hline
        MLMC  & $\epsilon^{-2}$  &  $\epsilon^{-1-q}$ &  $\epsilon^{-1-q}$ \\\hline
        MLMC-SR & $\epsilon^{-2}$  &  $\epsilon^{-2}$   &   $\epsilon^{-q}$ \\\hline
      \end{tabular}
    \end{center}
    \caption{Comparison of work between MC, MLMC with full refinement (MLMC), and MLMC with selective refinement (MLMC-SR) for different $q$.}
\label{tab:comparison}
\end{table}

\section{Heuristic algorithm}
\label{sec:heuristic}
In this section, we present a heuristic algorithm for the MLMC method
with selective refinement. Contrary to Theorem~\ref{thm:workMLMCsel},
this algorithm does not guarantee that the RMSE is
$\mathcal{O}(\epsilon)$, since we in practice lack a priori knowledge
of the constants $\ME$ and $\MV$ in Lemma~\ref{lem:assEVW}. Instead,
the RMSE needs to be estimated. Recall the split of the MSE into a
numerical and statistical contribution:
\begin{equation}
    \left(\EV{Q-\widehat Q}\right)^2\leq \frac{1}{2}\epsilon^2\quad\text{and}\quad \VV{\widehat Q} \leq \frac{1}{2}\epsilon^2.
\end{equation}
With $\widehat Q$ being the multilevel Monte Carlo estimator
$\QML{L}$, we here present heuristics for estimating the numerical and
statistical error of the estimator.

For both estimates and $\ell \ge 1$, we make use of the trinomially
distributed variable $Y_\ell(\omega) = Q_\ell(\omega) -
Q_{\ell-1}(\omega)$.  We denote the probabilities for $Y_\ell$ to be
$-1$, $0$ and $1$ by $p_{-1}$, $p_{0}$ and $p_{1}$, respectively. For
convenience, we drop the index $\ell$ for the probabilities, however,
they do depend on $\ell$. In order to estimate the numerical bias
$\EV{Q - \QML{L}} = \EV{Q - Q_L}$, we assume that ${\bf M1}$ holds
approximately with equality, i.e., $\left|\EV{Q-Q_\ell}\right| \approx
\ME\gamma^\ell$. Then the numerical bias can be overestimated, 
$|\EV{Q - Q_\ell}| \le |\EV{Y_\ell}|(\gamma^{-1}-1)^{-1}$, since
\begin{equation*}
  \begin{aligned}
    \left| \EV{Y_\ell} \right| & = \left| \EV{Q_\ell-Q} - \EV{Q_{\ell-1}-Q}  \right| \\
    & \ge \left| |\EV{Q_\ell-Q}| - |\EV{Q_{\ell-1}-Q}|  \right| \\
    & \approx \left| \ME \gamma^\ell - \ME \gamma^{\ell-1} \right| \\
    & = \ME \gamma^{\ell}(\gamma^{-1} - 1). \\
  \end{aligned}
\end{equation*}
Hence, we concentrate our effort on estimating $|\EV{Y_\ell}|$.

It has been observed that the accuracy of sample estimates of mean and
variance of $Y_\ell$ might deteriorate for deep levels $\ell \gg 1$,
and a continuation multilevel Monte Carlo method was proposed in
\cite{CHNST14} as a remedy for this. That idea could be applied and
specialized for this functional to obtain more accurate
estimates. However, in this work we use the properties of the
trinomially distributed $Y_\ell$ to construct a method with optimal
asymptotic behavior, possibly with increase of computational cost by a
constant.

We consider the three binomial distributions $[Y_\ell = 1]$,
$[Y_\ell = -1]$ and $[Y_\ell \ne 0]$ which have parameters $p_1$,
$p_{-1}$ and $p_1 + p_{-1}$, respectively ($[\cdot]$ is the Iverson
bracket notation). These parameters can be used in estimates for both
the expectation value and variance of the trinomially distributed
$Y_\ell$. Considering a general binomial distribution $B(n, p)$, we
want to estimate $p$. For our distributions, as the level $\ell$
increases, $p$ approaches zero, why we are concerned with finding
stable estimates for small $p$. It is important that the parameter is
not underestimated, since it is used to control the numerical bias and
statistical error and could then cause premature termination. We
propose an estimation method that is easy to implement, and that will
overestimate the parameter in case of accuracy problems, rather than
underestimate it, while keeping the asymptotic rates given in
Lemma~\ref{lem:assEVW} for the estimators.

The standard unbiased estimator of $p$ is $\hat p = xn^{-1}$, where
$x$ is the number of observed successes. The proposed alternative (and
biased) estimator is $\tilde p = (x + k)(n + k)^{-1}$ for a $k >
0$. This corresponds to a Bayesian estimate with prior beta
distribution with parameters $(k+1, 1)$.  Observing that
\begin{equation}
  \label{eq:ev_p}
  \begin{aligned}
    |\EV{Y_\ell}| & = |p_1 - p_{-1}|, \\
    \VV{Y_\ell} & = p_1 + p_{-1} - (p_1 - p_{-1})^2
  \end{aligned}
\end{equation}
and considering Lemma~\ref{lem:assEVW} (assuming equality with the
rates), we conclude that all three parameters $p \propto \gamma^\ell$
(where $\propto$ means asymptotically proportional to, for $\ell \gg
1$). With the standard estimator $\hat p$, the relative variance can
be expressed as $\VV{\hat p}(\EV{\hat p})^{-2}$. This quantity should
be less than one for an accurate estimate. We now examine its
asymptotic behavior.  The parameter $n$ is the optimal number of
samples at level $\ell$ (equation \eqref{eq:MLMCoptimalNl}) and can be
expressed as
\begin{equation}
  \label{eq:n_asymptotic}
  n \propto \gamma^{\frac{1}{2}\ell q - \frac{1}{2}L(2+q)},
\end{equation}
where we used that $\epsilon \propto \gamma^L$, $\cost{Y_\ell} \propto
\gamma^{(1-q)\ell}$ and $\VV{Y_\ell} \propto \gamma^\ell$.  Then we
have
\begin{equation*}
  \begin{aligned}
    \frac{\VV{\hat p}}{\EV{\hat p}^{2}} & = \frac{n^{-1}p(1-p)}{p^2} = \frac{1-p}{np}
     \propto \gamma^{\frac{2+q}{2}(L-\ell)}.
  \end{aligned}
\end{equation*}
In particular, for $\ell = L$, the relative variance is asymptotically
constant, but we don't know a priori how big this constant is. When it
is large (greater than $1$), the relative variance of $\hat p$ might
be very large. An analogous analysis on $\tilde p$ yields
\begin{equation}
  \label{eq:relative_variance}
  \begin{aligned}
    \frac{\VV{\tilde p}}{\EV{\tilde p}^{2}} & = \frac{(n+k)^{-2}np(1-p)}{(n+k)^{-2}(np + k)^2} = \frac{np(1-p)}{(np+k)^2} \le \frac{np}{(np+k)^2}. \\
  \end{aligned}
\end{equation}
Maximizing the bound in \eqref{eq:relative_variance} with respect to $np$, gives
\begin{equation*}
  \begin{aligned}
    \frac{\VV{\tilde p}}{\EV{\tilde p}^{2}} & \le \frac{1}{4k}. \\
  \end{aligned}
\end{equation*}
Choosing for instance $k = 1$ gives a maximum relative variance of
$1/4$. Choosing a larger $k$ gives larger bias, but smaller relative
variance. The bias of this estimator is significant if $np \ll k$,
however, that is the case when we have too few samples to estimate the
parameter accurately, and then $\tilde p$ instead acts as a bound. The
estimate $\tilde p$ keeps the asymptotic behavior $\EV{\tilde p}
\propto \gamma^\ell$, since
\begin{equation*}
  \begin{aligned}
    \EV{\tilde p} & = \frac{np + k}{n+k} \propto \frac{np + k}{n} = p + \frac{k}{n} \propto p,
  \end{aligned}
\end{equation*}
where we use that $n$ dominates $k$ for large $\ell$.

Now, estimating the parameters $p_1$, $p_{-1}$ and $p_{1} + p_{-1}$ as
$\tilde p_1$, $\tilde p_{-1}$ and $\tilde p_{\pm 1}$, respectively,
using the estimator $\tilde p$ above (note that the sum $p_1 + p_{-1}$
is estimated separately from $p_1$ and $p_{-1}$) we can bound
(approximately) the expected value and variance of $Y_\ell$ in \eqref{eq:ev_p}:
\begin{equation}
  \label{eq:expected_p}
  \begin{aligned}
    |\EV{Y_\ell}| & \le \max(p_1, p_{-1}) \approx \max(\tilde p_1, \tilde p_{-1})
  \end{aligned}
\end{equation}
and
\begin{equation}
  \begin{aligned}
  \label{eq:variance_p}
  \VV{Y_\ell} & \le p_1 + p_{-1} \approx \tilde p_{\pm 1}
  \end{aligned}
\end{equation}
for $\ell \ge 1$. For $\ell = 0$, the sample size is usually large
enough to use the sample mean and variance as accurate
estimates. Since the asymptotic behavior of $\tilde p$ is
$\gamma^\ell$, the rates in Lemma~\ref{lem:assEVW} still holds and
Theorem~\ref{thm:workMLMCsel} applies (however, with approximate
quantities).

The algorithm for the MLMC method using selective refinement is
presented in Algorithm~\ref{alg:MLMCselective}. The termination
criterion is the same as was used in the standard MLMC algorithm
\cite{Gil08}, i.e.,
\begin{equation}
  \label{eq:termination}
  \begin{aligned}
    \max(\gamma |\EV{Y_{L-1}}|, |\EV{Y_{L}}|) < \frac{1}{\sqrt{2}}(\gamma^{-1} - 1)\epsilon,
  \end{aligned}
\end{equation}
where $|\EV{Y_{L-1}}|$ and $|\EV{Y_{L}}|$ are estimated using the
methods presented above. A difference from the standard MLMC algorithm
is that the initial sample size for level $L$ is $N_L = N\gamma^{-L}$
instead of $N_L = N$, for some $N$. This is what is predicted by
equation \eqref{eq:n_asymptotic} and is necessary to provide accurate
estimates of the expectation value and variance of $Y_\ell$ for deep
levels. Other differences from the standard MLMC algorithm is that the
selective refinement algorithm (Algorithm~\ref{alg:selective}) is used
to compute $\QMC{L}$, and that the estimates of expectation value and
variance of $Y_\ell$ are computed according to the discussion above.

\begin{algorithm}[htb]\caption{MLMC method using selective refinement}
\label{alg:MLMCselective}
\begin{algorithmic}[1]
  \State Pick critical value $y$, cost model parameter $q$, tolerance factor $\gamma$, initial number of samples $N$, parameter $k$, and final tolerance $\epsilon$
  \State Set $L=0$
  \Loop
  \State Let $N_L = N\gamma^{-L}$ and compute $\QMC{L}$ using selective refinement (Algorithm~\ref{alg:selective})
  \State Estimate $\VV{Q_\ell-Q_{\ell-1}}$ using \eqref{eq:expected_p}
  \State Estimate the optimal $\{N_\ell\}_{\ell=0}^L$ using \eqref{eq:MLMCoptimalNl} and cost model \eqref{eq:costmodel_selective}
  \State \parbox[t]{\dimexpr\linewidth-\algorithmicindent}{Compute $\QMC{\ell}$ for all levels $\ell = 0,\ldots,L$ using selective refinement (Algorithm~\ref{alg:selective}) \strut}
  \State Estimate $\EV{Q_\ell-Q_{\ell-1}}$ using \eqref{eq:variance_p}
  \State Terminate if converged by checking inequality \eqref{eq:termination}
  \State Set L = L + 1
  \EndLoop
  \State The MLMC-SR estimator is $\QML{L} = \sum_{\ell=1}^L \QMC{\ell}$
\end{algorithmic}
\end{algorithm}

\section{Numerical experiments}
\label{sec:numerical}
Two types of numerical experiments are presented in this section. The
first experiment (in Section~\ref{sec:demonstrational}) is performed
on a simple and cheap model $\mathcal{M}$ so that the asymptotic
results of the computational cost, derived in Theorem~\ref{thm:workMLMCsel},
can be verified. The second experiment (in
Section~\ref{sec:single-phase}) is performed on a PDE model
$\mathcal{M}$ to show the method's applicability to realistic
problems. In our experiments we made use of the software FEniCS
\cite{FenicsBook} and SciPy \cite{SciPy}.

\subsection{Failure probability of a normal distribution}
\label{sec:demonstrational}
In this first demonstrational experiment, we let the quantity of
interest $X$ belong to the standard normal distribution and we seek to
find the probability of $X \le y = 0.8$. The true value of this
probability is $\Pr(X \le 0.8) = \Phi(0.8) \approx 0.78814$ and we
hence have a reliable reference solution. We define approximations
$X_h$ of $X$ as follows. First, we let our input data $\omega$ belong
to the standard normal distribution, and let $X(\omega) =
\omega$. Then, we let $X_h(\omega) = \omega +
h(2U(\omega,h)-1+b)/(1+b)$, where $b = 0.1$ and $U(\omega,h)$ is a
uniformly distributed random number between $0$ and $1$. Since we have
an error bound $|X_h - X| \le h$, the selective refinement algorithm
(Algorithm~\ref{alg:selective}) can be used to construct a function
$X_\ell$ satisfying Assumption~\ref{ass:qoierror}. With this setup it
is very cheap to compute $X_h$ to any accuracy $h$, however, for
illustrational purposes we assume a cost model $\cost{X_h} = h^{-q}$
with $q = 1$, $2$, and $3$ to cover the three cases in
Theorem~\ref{thm:workMLMCsel}.

For the three values of $q$, and eight logarithmically distributed
values of $\epsilon$ between $10^{-3}$ and $10^{-1}$, we performed
$100$ runs of Algorithm~\ref{alg:MLMCselective}. All parameters used
in the simulations are presented in Table~\ref{tab:demonstrational}.
\begin{table}[ht]
  \begin{center}
    \begin{tabular}{|c|c|c|c|} \hline
      Parameter  & Value \\ \hline\hline
      $y$ & $0.8$ \\ \hline
      $q$ & $1$, $2$, $3$ \\ \hline
      $\gamma$ & $0.5$ \\ \hline
      $N$ & $10$ \\ \hline
      $k$ & $1$ \\ \hline
      $\epsilon$ & $(10^{-3}, 10^{-1})$ \\ \hline
    \end{tabular}
  \end{center}
  \caption{Parameters used for the demonstrational experiment.}
  \label{tab:demonstrational}
\end{table}

For convenience, we denote by $\widehat Q_i$ the MLMC-SR estimator
$\QML{L}$ of the failure probability from run $i = 1,\ldots,M$ with $M
= 100$. For each tolerance $\epsilon$ and cost parameter $q$, we
estimated the RMSE of the MLMC-SR estimator by
\begin{equation*}
  \rmse{\QML{L}}
  = \left(\EV{\left(\QML{L}- \EV{Q}\right)^2}\right)^{1/2}
 \approx \left(\frac{1}{M}\sum_{i=1}^{M}\left(\widehat Q_i - \EV{Q}\right)^2\right)^{1/2}.
\end{equation*}
Also, for each of the eight tolerances $\epsilon$, we computed the
run-specific estimation errors $|\widehat Q_i - \EV{Q}|$,
$i=1,\ldots,M$. In Figure~\ref{fig:rmse} we present three plots of the
RMSE vs.\ $\epsilon$, one for each value of $q$. We can
see that the method yields solutions with the correct accuracy.
\begin{figure}[ht]
  \centering
  \begin{subfigure}[b]{0.49\textwidth}
    \includegraphics[]{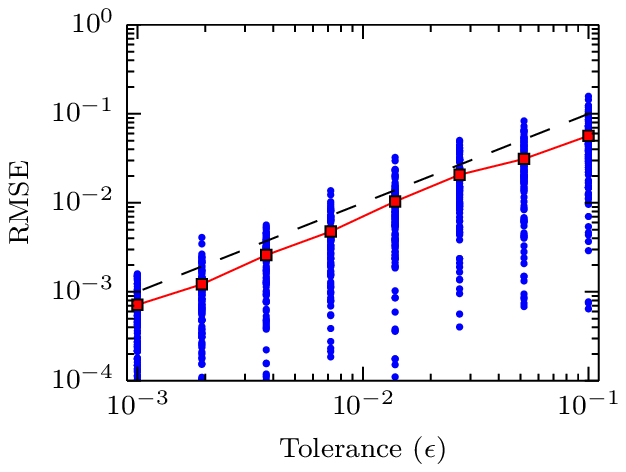}
    \caption{Case q = 1.}
  \end{subfigure}
  \begin{subfigure}[b]{0.49\textwidth}
    \includegraphics[]{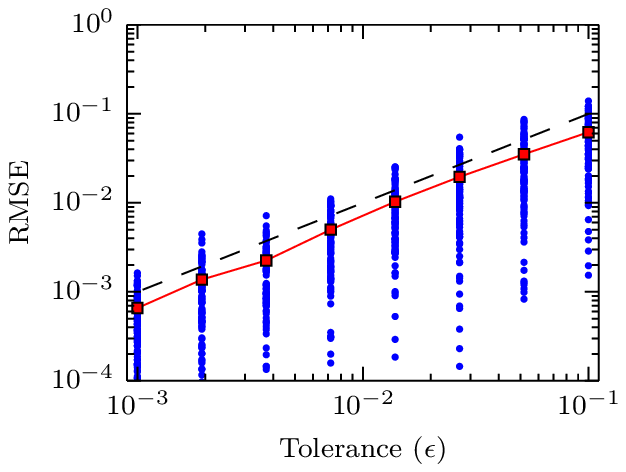}
    \caption{Case q = 2.}
  \end{subfigure}
  \begin{subfigure}[b]{0.49\textwidth}
    \includegraphics[]{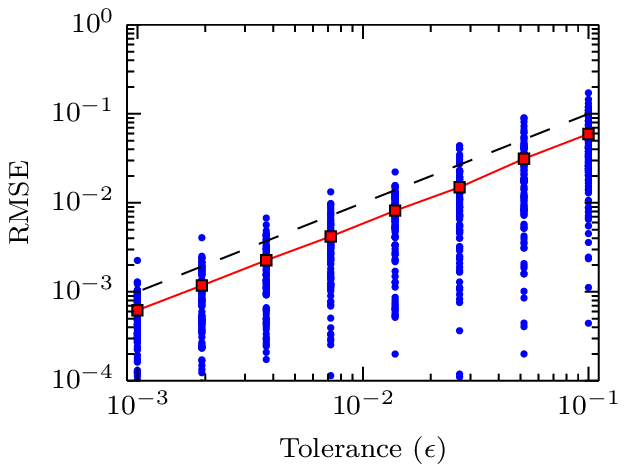}
    \caption{Case q = 3.}
  \end{subfigure}
  \begin{subfigure}[b]{0.49\textwidth}
    \includegraphics[]{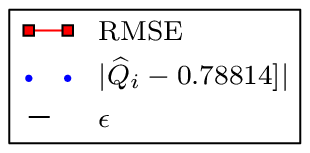}
    \caption{Legend.}
  \end{subfigure}
  \caption{RMSE (square markers and line) plotted vs.\ tolerance for
    the experiment described in Section~\ref{sec:demonstrational}. The
    dashed line is the tolerance $\epsilon$ and the dots are the
    individual errors for the $100$ runs at each
    tolerance.}\label{fig:rmse}
\end{figure}

In order to verify Theorem~\ref{thm:workMLMCsel}, we estimated the
expected cost for each tolerance $\epsilon$ and value of $q$ by
computing the mean of the total cost over the $100$ runs. The cost for
each realization was computed using the cost model in equation
\eqref{eq:costmodel}. The cost for realizations differs not only
between levels $\ell$, but also within a level $\ell$ owing to the
selective refinement algorithm. For each run $i$, the costs of all
realizations were summed to obtain the total cost for that run. We
computed a mean of the total costs for the $100$ runs. A plot of the
result can be found in Figure~\ref{fig:work}. As the tolerance
$\epsilon$ decreases the expected cost approaches the rates given in
Theorem~\ref{thm:workMLMCsel}. The reference costs are multiplied by
constants to align well with the estimated expected costs.
\begin{figure}[ht]
  \centering
    \includegraphics[]{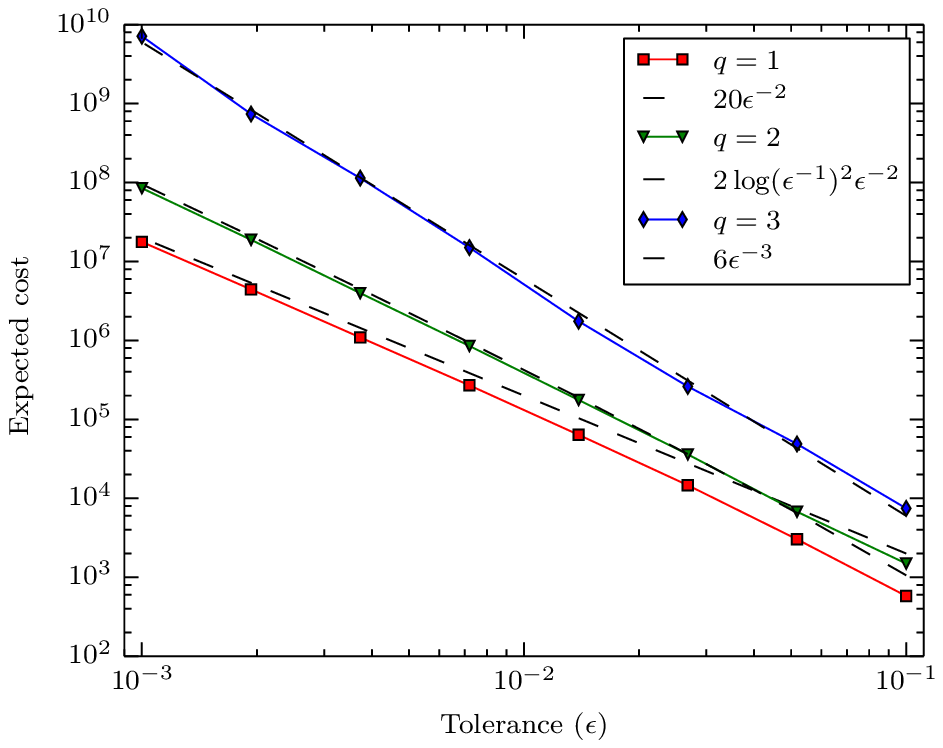}
    \caption{Computed mean total cost (diamond, triangle, square
      markers and lines) plotted with theoretical reference cost
      (dashed lines) for the experiment described in
      Section~\ref{sec:demonstrational}. The reference costs for the
      three values of $q$ are: $20\epsilon^{-2}$ for $q=1$;
      $2\log(\epsilon^{-1})^2\epsilon^{-2}$ for $q=2$; and
      $6\epsilon^{-3}$ for $q=3$.}\label{fig:work}
\end{figure}

\subsection{Single-phase flow in media with lognormal permeability}
\label{sec:single-phase}
We consider Darcy's law on a unit square $[0,1]^2$
on which we have impearmeable upper and lower boundaries, high
pressure on the left boundary ($\Gamma_1$) and low pressure on the
right boundary ($\Gamma_2$). We define the spaces $H^1_f(\domain)
= \{v \in H^1(\domain) : v|_{\Gamma_1} = f \text{ and } v|_{\Gamma_2}
= 0\}$, and let $n$ denote the unit normal of
$\domain$.

The weak form of the partial differential equation reads: find $u \in H^1_1(\domain)$ such that
\begin{equation}
  (a(\omega,\cdot) \nabla u, \nabla v) = 0\quad \text{ in } \domain,
\end{equation}
for all $v \in H_0^1(\domain)$, and $a$ is a stationary log-normal
distributed random field
\begin{equation}
  a(\omega,\cdot) = \exp(\kappa(\omega,\cdot)),
\end{equation}
over $\domain$, where $\kappa(\cdot,x)$ has zero mean and is normal
distributed with exponential covariance, i.e., for all $x_1, x_2 \in
\domain$ we have that
\begin{equation}
  \VV{\kappa(\cdot,x_1)\kappa(\cdot,x_2)} = \sigma^2\exp\left( \frac{-\|x_1-x_2\|_2}{\rho} \right).
\end{equation}
We choose $\sigma = 1$ and $\rho=0.1$ in the numerical experiment.

We are interested in the boundary flux on the right boundary, i.e.,
the functional $X(\omega) = \int_{\Gamma_2} n \cdot a(\omega)\nabla u
\dd x = (a(\omega,\cdot) \nabla u, \nabla g)$, for any $g \in
H^1(\domain)$, $g|_{\Gamma_1} = 0$ and $g|_{\Gamma_2} = 1$. The last
equality comes by a generalized Green's identity, see
\cite[Chp.~1,\ Corollary~2.1]{GiRa86}.

To generate realizations of $a(\omega,\cdot)$, the circulant embedding
method introduced in \cite{DiNe97} is employed.  The mesh resolution
for the input data of the realizations generated on level $\ell$ in
the MLMC-SR algorithm is chosen such that the finest mesh needed on
level $\ell$ is not finer than the chosen mesh. For a fixed
realization on level $\ell$ we don't know how fine data we need,
because of the selective refinement procedure. This means that the
complexity obtained for the MLMC-SR algorithm do not apply for the
generation of data. The circulant embedding method has log-linear
complexity. A remedy for the complexity of generating realizations is
to use a truncated Karhunen-Loève expansion that can easily be
refined. However, numerical experiments show that we are in a regime
where the time spent on generating realizations using circulant
embedding is negligible compared to the time spent in the linear
solvers.

The PDE is discretized using a FEM-discretization with linear Lagrange
elements. We have a family of structured nested meshes
$\mathcal{T}_{h_m}$, where a mesh $h_m$ is the maximum element
diameter of the given mesh. The data $a(\omega,\cdot)$ is defined in
the grid points of the meshes. Using the circulant embedding we get an
exact representation of the stochastic field in the grid points of the
given mesh. This can be interpreted as not making any approximation of
the stochastic field but instead making a quadrature error when
computing the bilinear form.

The functional for a discretization on mesh $m$ is defined as
$X_{h_{m}}(\omega) = (a(\omega,\cdot) \nabla u_{h_m}, \nabla g)$. The
convergence rates in energy norm for log-normal data is
$h^{1/2-\delta}$ for any $\delta > 0$ \cite{ChScTe13}. Using
postprocessing, it can be shown that the error in the functional
converges twice as fast \cite{GiSu02}, i.e,
$|X_{h_{m}}-X_{h_{m}}(\omega)|\leq Ch^{s-2\delta}$ for $s=1$. We use a
multigrid solver that has linear $\alpha=1$ (up to $\log$-factors)
complexity. The work for one sample can then be computed as
$\gamma^{-q\ell}$ where $\gamma^{\ell}$ is the numerical bias
tolerance for the sample and $q\approx 2\alpha/s = 2$, which was also
verified numerically. The error is estimated using the dual solution
computed on a finer mesh. Since it can be quite expensive to solve a
dual problem for each realization of the data, the error in the
functional can also be computed by estimating the constant $C$ and $s$
either numerically or theoretically.

We choose $\gamma=0.5$, $N=10$, and $k=1$ in the the MLMC-SR
algorithm, see Section~\ref{sec:heuristic} for more
information on the choices of parameters.  The problem
reads: find the probability $p$ for $X\leq y=1.5$ to the given RMSE
$\epsilon$. We compute $p$ for
$\epsilon=10^{-1}$, $10^{-1.5}$, and $10^{-2}$. All parameters used in the
simulation are presented in Table~\ref{tab:lognomal}.
\begin{table}[tb]
  \begin{center}
    \begin{tabular}{|c|c|} \hline
      Parameter  & Value \\ \hline\hline
      $y$ & $1.5$ \\ \hline
      $q$ & $2$ \\ \hline
      $\gamma$ & $0.5$ \\ \hline
      $N$ & $10$ \\ \hline
      $k$ & $1$ \\ \hline
      $\epsilon$ & $10^{-1},10^{-1.5},10^{-2}$ \\ \hline
      $\rho$ & $ 0.1$ \\ \hline
      $\sigma$ & $ 1$ \\ \hline
    \end{tabular}
  \end{center}
  \caption{Parameters used for the single-phase flow experiment. The parameters $y,q,\gamma,N,k,\epsilon$ are used in the MLMC-SR algorithm and $\rho$, $\sigma$ to define the log-normal field.}
  \label{tab:lognomal}
\end{table}
To verify the accuracy of the estimator we compute $100$ simulations
of the MLMC-SR estimator for each RMSE $\epsilon$ and present the
sample standard deviation (square root of the sample variance) of the
MLMC-SR estimators in Table~\ref{tab:lognormal-result}.
\begin{table}[tb]
  \begin{center}
    \begin{tabular}{|c|c|c|c|} \hline
      $\epsilon$ & Mean $p$  & Sample std           & Target std ($\epsilon/\sqrt{2}$)  \\ \hline\hline
      $10^{-1}$   & $0.8834$  & $6.472\cdot10^{-2}$   & $7.071\cdot10^{-2}$   \\ \hline
      $10^{-1.5}$ &  $0.8890$ & $1.873\cdot10^{-2}$   & $2.236\cdot10^{-2}$   \\ \hline
      $10^{-2}$   & $0.8933$  & $5.557\cdot10^{-2}$   & $7.071\cdot10^{-3} $  \\ \hline
    \end{tabular}
  \end{center}
  \caption{The mean failure probability $p$ and sample standard deviation (std) is computed using 100 MLMC-SR estimators and compared to the target std which is the statistical part of the RMSE error $\epsilon$.}
  \label{tab:lognormal-result}
\end{table}
We see that in all the three cases the sample standard deviation is
smaller than the statistical contribution $\epsilon/\sqrt{2}$ of the
RMSE $\epsilon$. Since the exact flux is unknown, the numerical
contribution in the estimator has to be approximated to be less than
$\epsilon/\sqrt{2}$ as well, which is done in the termination
criterion of the MLMC-SR algorithm so it is not presented here. The
mean number of samples computed to the different tolerances on each
level of the MLMC-SR algorithm is computed from 100 simulations of the
MLMC-SR estimator for $\epsilon = 10^{-2}$ and are shown in
Table~\ref{tab:work}.
\begin{table}[tb]
  \begin{center}
    \begin{tabular}{|c|c|c|c|c|c|} \hline
      $\ell$  & $0$  & $1$ & $2$ & $3$ & $4$ \\ \hline\hline
      Mean $N_\ell$ & $16526.81$ & $9045.41$ & $4524.83$ & $1471.63$ & $738.63$ \\ \hline\hline
      $j = 0$ & $16526.81$ &    $4520.99$ &  $2265.23$ &   $734.21$  &  $366.9$  \\ \hline
      $j = 1$ &            &    $4524.42$ &  $1486.62$ &   $484.11$  &  $244.69$ \\ \hline
      $j = 2$ &            &              &  $772.98$  &   $232.33$  &  $116.77$ \\ \hline
      $j = 3$ &            &              &            &   $20.98$   &    $9.76$ \\ \hline
      $j = 4$ &            &              &            &             &    $0.51$ \\ \hline
    \end{tabular}
  \end{center}
  \caption{The distribution of realizations solved to different tolerance levels $j$ for the case $\epsilon = 10^{-2}$. The table is based on the mean of $100$ runs.}
  \label{tab:work}
\end{table}
The table shows that the selective refinement algorithm only refines a
fraction of all problems to the highest accuracy level $j =
\ell$. Using a MLMC method (without selective refinement) $N_\ell$
problem would be solved to the highest accuracy level. Using the cost
model $\gamma^{-q\ell}$ for $\epsilon = 10^{-2}$ we gain a factor
$\sim 6$ in computational cost for this particular problem using
MLMC-SR compared to MLMC. From Theorem~\ref{thm:workMLMCsel} the
computational cost for MLMC-SR and MLMC increase as
$\epsilon^{-2}\log(\epsilon^{-1})^2$ and $\epsilon^{-3}$, respectively.

\appendix

\section{Derivation of optimal level sample size}\label{optimalNl}
To determine the optimal sample level size $N_\ell$ in equation
\eqref{eq:MLMCoptimalNl}, we minimize the total cost keeping the
variance of the MLMC estimator equal to $\epsilon^2/2$, i.e.,
\begin{equation}
  \begin{aligned}
    \text{min} & \sum_{\ell=0}^L N_\ell c_\ell \\
    \text{subject to} & \sum_{\ell=0}^LN_\ell^{-1}\VV{Y_\ell} = \epsilon^2/2,
  \end{aligned}
\end{equation}
where $Y_\ell = Q_\ell - Q_{\ell-1}$. We reformulate the problem using a Lagrangian multiplier $\mu$ for the
constraint. Define the objective function
\begin{equation}
  g(N_\ell,\mu) = \sum_{\ell=0}^L N_\ell c_\ell + \mu\left(\sum_{\ell=0}^LN_\ell^{-1}\VV{Y_\ell} - \epsilon^2/2\right).
\end{equation}
The solution is a stationary point $(N_\ell,\mu)$ such that
$\nabla_{N_\ell,\mu}g(N_\ell,\mu)=0$. Denoting by $\hat N_\ell$ and
$\hat \mu$ the components of the gradient, we obtain
\begin{equation}
  \nabla_{N_\ell,\mu}g(N_\ell,\mu)=\left(c_\ell - \mu N_\ell^{-2}\VV{Y_\ell}\right)\hat N_\ell + \left(\sum_{\ell=0}^LN_\ell^{-1}\VV{Y_\ell} - \epsilon^2/2\right)\hat \mu.
\end{equation}
Choosing $N_\ell= \sqrt{\mu\VV{ Y_\ell}/ c_\ell}$ makes the $\hat
N_\ell$ components zero. The $\hat \mu$ component is zero when
$\sum_{\ell=0}^LN_\ell^{-1}\VV{Y_\ell}= \epsilon^2/2$. Plugging in
$N_\ell$ yields $
2\epsilon^{-2}\sum_{\ell=0}^L\sqrt{\VV{Y_\ell}c_\ell}= \sqrt{\mu}$ and
hence the optimal sample size is
\begin{equation}
  N_\ell = 2\epsilon^{-2}\sqrt{\VV{ Y_\ell}/ c_\ell}\sum_{k=0}^L\sqrt{\VV{Y_k}c_k}.
\end{equation}

 \bibliographystyle{plain}
 \bibliography{references}

\end{document}